\documentclass[11 pt, final]{article}
\title{Lamplighters admit weakly aperiodic SFTs.}
\author{David Bruce Cohen}

\usepackage{fullpage}
\usepackage[latin1]{inputenc}
\usepackage{tikz}
\usetikzlibrary{shapes,arrows}
\usepackage{graphicx}
\usepackage{amsmath}
\usepackage{amssymb}
\usepackage{amsthm}
\usepackage{epsfig,pinlabel}
\usepackage{amsfonts}
\usepackage{amscd}

\def\ZZ{{\mathbb{Z}}}
\def\ZZgo{{\mathbb{Z}_{\geq 0}}}
\DeclareMathOperator{\NN}{\ensuremath{\mathbb{N}}}

\def\Aut{{\text{Aut}}}

\def\To{{\rightarrow}}

\def\la{{\langle}}
\def\ra{{\rangle}}

\def\Maj{{\text{Maj}}}
\def\nnn{{\mathfrak{n}}}
\def\pil{{\pi_\ell}}
\def\Cong{{\text{Con}_G}}
\def\Stab{{\text{Stab}}}
\def\zeroone{{\{0,1\}}}
\def\Gam{{\Gamma}}
\def\reppy{{\bigoplus_{\ZZ}\Lambda}}
\def\repmin{{\bigoplus_{-\mathbb{N}}\Lambda}}
\def\Lam{{\Lambda}}
\def\lam{{\lambda}}
\def\muv{{{\buildrel \rightarrow\over\mu}}}

\def\sig{{\sigma}}
\def\sigo{{\sigma_0}}
\def\rolemod{{\mathfrak{RM}}}
\def\ellint{{\left\{0,\ldots,\ell-1\right\}}}

\theoremstyle{plain}
\newtheorem{theorem}{Theorem}[section]
\newtheorem{lemma}[theorem]{Lemma}

\newtheorem{proposition}[theorem]{Proposition}

\newtheorem{definition}[theorem]{Definition}

\begin{document}

\maketitle

\begin{abstract}
Let $A$ be a finite set and $G$ a group. A closed subset $X$ of $A^G$ is called a subshift if the action of $G$ on $A^G$ preserves $X$. If $K$ is a closed subset of $A^G$ such that membership in $K$ is determined by looking at a fixed finite set of coordinates, and $X$ is the intersection of all translates of $K$ under the action of $G$, then $X$ is called a subshift of finite type (SFT). If an SFT is nonempty and contains no finite $G$-orbits, it is said to be weakly aperiodic. A virtually cyclic group has no weakly aperiodic SFT, and Carroll and Penland have conjectured that a group with no weakly aperiodic SFT must be virtually cyclic. Answering a question of Jeandel, we show that lamplighters always admit weakly aperiodic SFTs.
\end{abstract}

\section{Introduction.}
\paragraph{Subshifts.} Let $A$ be a discrete finite set. For any set $F$, let $A^F$ have the usual meaning of $\{\sig:F\To A\}$---that is $A^F$ is the set of all functions from $F$ to $A$. The full $A$-shift on a group $G$ is the set $A^G$ equipped with the product topology (equivalently, the topology of pointwise convergence) and the right $G$-action given by $(\sig\cdot g)(h):=\sig(gh)$. If a closed subset $X\subseteq A^G$ is preserved by the $G$-action, then $X$ is called a {\bf subshift.}

\paragraph{SFTs.} The simplest means of producing a subshift is as follows. A {\bf pattern} is a function from a finite subset $F\subseteq G$ to $A$. Let $p_1:F_1\To A,\ldots,p_n:F_n\To A$ be patterns, and let $K\subseteq A^G$ consist of all $\sig\in A^G$ such that for $i=1,\ldots,n$, we have that the restriction $\sig|_{F_i}$ is not identically equal to $p_i$---in other words, for $\sig$ to be in $K$, we must have, for each $i$, some $g_i\in F_i$ such that $\sig(g_i)\neq p_i(g_i)$. Then $X:=\bigcap_{g\in G} K\cdot g$ is called a {\bf subshift of finite type} (SFT). More specifically, we say that $X$ is the SFT carved out by forbidden patterns $p_1,\ldots,p_n$.

\paragraph{Weak aperiodicity.} An SFT $X\subseteq A^G$ is said to be {\bf weakly aperiodic} if it is nonempty and does not contain any configuration $\sig\in A^G$ which is fixed by a finite index subgroup of $G$. It is not hard to see that there are no weakly aperiodic SFTs over the group $\ZZ$: if $X\subseteq A^\ZZ$ is a nonempty SFT, then it must contain some element $\sig_0$. Considered as a bi-infinite word, $\sig_0$ must contain two disjoint copies of some word $w$ which is longer than all the forbidden patterns used to define $X$. If the minimal subword of $\sig_0$ containing these instances of $w$ is written as a concatenation of words $wvw$, then it is clear that $\ldots wvwvwv\ldots$ represents a periodic element of $X$.

\begin{figure}[t]
\labellist
\small\hair 2pt

\endlabellist

\centering
\centerline{\psfig{file=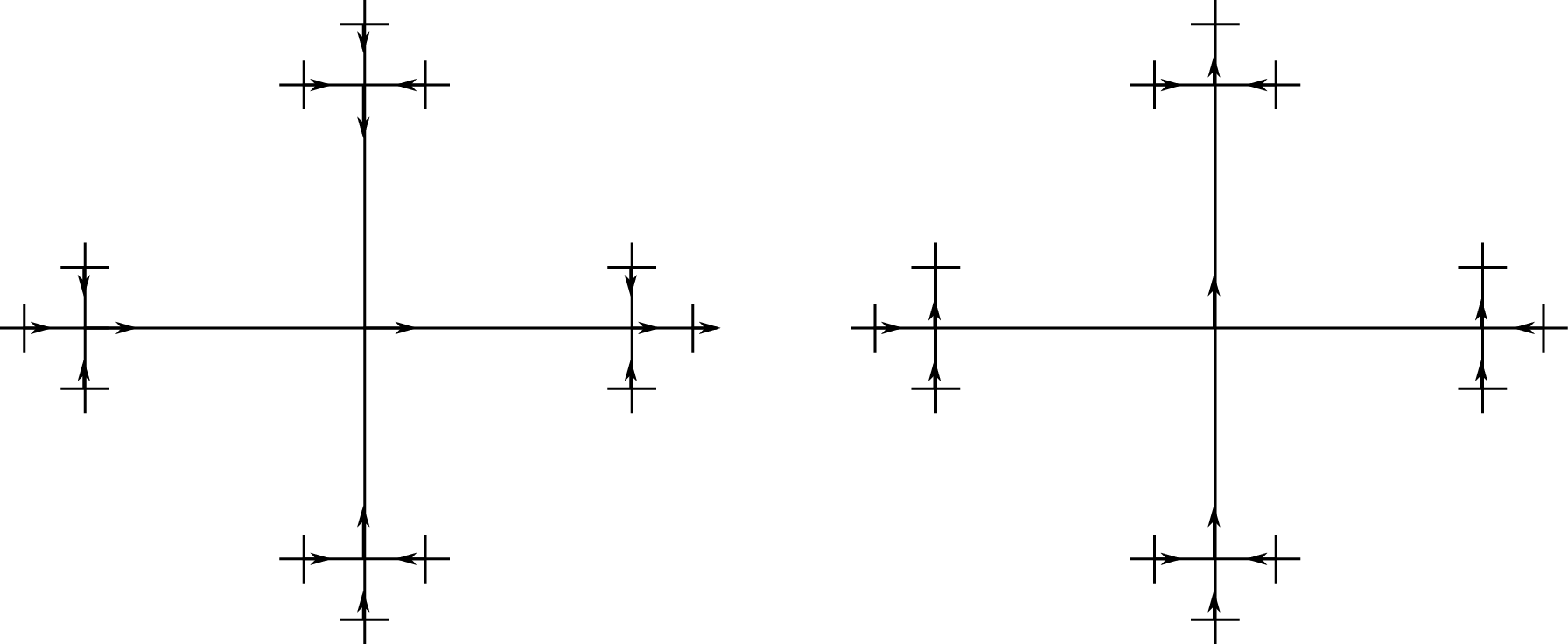,scale=90}}
\caption{The figure on the left depicts a Ponzi scheme, the figure on the right does not, as there is only one arrow pointing to the central vertex. The space of Ponzi schemes on the free group is a weakly aperiodic SFT.}
\label{figure:ponzi}
\end{figure}

On the other hand, one can construct a weakly aperiodic SFT on the free group $\la a|\ra\ast\la b|\ra$ as follows. Let the alphabet $A$ equal $\left\{a,a^{-1},b,b^{-1}\right\}$, and let $X\subseteq A^G$ consist of all configurations $\sig$ such that, for every $g\in G$ we have
$$\#\left\{s\in A: gs\sig\left(gs\right)=g\right\}\geq 2.$$
Given $\sigma\in A^G$ we may draw, for each $g\in G$, an arrow pointing from $g$ to $g\sigma(g)$ in the Cayley graph. In this interpretation, $\sigma$ being in $X$ is equivalent to the condition that every vertex of the Cayley graph is pointed to by at least two other vertices. Such an assignment of arrows to vertices of a graph is called a Ponzi scheme. One may construct elements of $X$ by hand (see Figure \ref{figure:ponzi}), and no element of $X$ can be fixed by a finite index subgroup because finite graphs do not have Ponzi schemes.

\paragraph{The conjecture of Carroll and Penland.} It is natural to ask which groups admit weakly aperiodic SFT. Carroll and Penland \cite[\S 4.1]{carpen} have conjectured that a group which is not virtually cyclic must admit a weakly aperiodic SFT. It is know that a counterexample to this conjecture---that is, a non virtually cyclic group with no weakly aperiodic SFT---must have a number of properties.
 
\paragraph{Restrictions on possible counterexamples.} Let $G$ be a group with no weakly aperiodic SFT. We now list some facts known about $G$.
\begin{itemize}
\item $G$ must be residually finite \cite[Proposition 3.2]{jeandel}. To see this, consider, for each $n\in\NN$ the SFT $X$ consisting of all $\sig\in A^G$ such that $\sig(g)\neq\sig\left(g'\right)$ for any $g'$ in the $n$-ball around $g$ in the Cayley graph of $G$. If $A$ is sufficiently large (depending on $n$), this subshift is nonempty. If some $\sig\in X$ is fixed by a finite index subgroup $\Gam\subseteq G$, then $\Gam$ has trivial intersection with the $n$-ball around the identity in $G$. Letting $n$ go to infinity, we see that every non-identity element of $g$ is excluded by some finite index subgroup.
\item Jeandel has shown that for all $n\in\NN$, $G$ must have finite index subgroup of index divisible by $\NN$ \cite[Corollary 3.3]{jeandel}. See also \cite{mn}.
\item Jeandel has shown that $G$ must be amenable \cite[Corollary 3.1]{jeandel}. See also \cite{bw}.
\item $G$ cannot be isomorphic to any Baumslag-Solitar group $BS(m,n)$ \cite[Theorem 8]{ak} or to $\ZZ^2$ \cite{berger}. Explicit examples of weakly aperiodic SFTs are known for those cases.
\item If $G$ is infinite and finitely presented, it must virtually surject onto $\ZZ$. We will now sketch the proof of this fact. \cite[\S 3]{thesis} shows that the space of ``derivatives" of $1$-Lipschitz functions from $G$ to $\ZZ$ is a subshift of finite type, where $G$ is endowed with the word metric with respect to some fixed finite generating set. We say that a $1$-Lipschitz function $f:G\To\ZZ$ has no local minimum if for all $g\in G$, there is some $g'$ at distance $1$ from $g$ in the Cayley graph such that $f\left(g'\right)=f(g)-1$. It is clear then that the space of derivatives of $1$-Lipschitz functions $G\To\ZZ$ with no local minimum is also an SFT. This SFT is nonempty because it contains points in the orbit closure of the derivative of the distance-from-the-identity function $g\mapsto d(g,1_G)$---this function has a unique local minimum, and thus the orbit closure of its derivative contains derivatives of $1$-Lipschitz functions with no local minimum. Since every nonempty SFT over $G$ contains an element fixed by some finite index subgroup, it follows that some finite index subgroup of $G$ must homomorphically surject onto $\ZZ$.
\item If $G$ is finitely presented, it must be QI-rigid in the sense described in \cite[Theorem 1.10]{thesis}. Namely, if a group $H$ is quasi-isometric to $G$, then for some finite subgroup $K\subseteq G$, we must have that $G/K$ is isomorphic to a finite index subgroup of $H$.
\item No subgroup of $G$ admits a weakly aperiodic SFT, since the intersection of a subgroup $H$ of $G$ with a finite index subgroup $\Gam$ of $G$ will be finite index in $H$. It is unclear to whom this fact should be attributed.
\item In fact, Jeandel has shown that no finitely presented group $H$ acting translation-like on $G$ admits a weakly aperiodic SFT \cite[Theorem 3]{jeandel}. A translation-like action of $H$ on $G$ is a free action by maps at a finite distance from the identity in the uniform metric. This notion was introduced by Whyte \cite{whyte} for the purpose of generalizing subgroups---if $H$ is a subgroup of $G$, then $H$ acts translation like on $G$ via $h\cdot g=gh^{-1}$ for $h\in H$ and $g\in G$. Jeandel's theorem, together with the fact that $\ZZ^2$ admits no weakly aperiodic SFT, implies that $G$ cannot be isomorphic to the direct product of two infinite finitely generated groups.
\item $G$ cannot be commensurable to a group with a weakly aperiodic SFT \cite[Theorem 11]{carpen}. 
\end{itemize}

\paragraph{Lamplighters.} Lamplighters---i.e., wreath products of $\ZZ$ with finite groups---comprise perhaps the only well known class of groups which have all the above properties. Jeandel has asked whether lamplighters have weakly aperiodic SFTs \cite[\S 5]{jeandel}. In this paper, we give a positive answer.

\begin{theorem}
\label{theorem:main}
Every lamplighter admits a weakly aperiodic SFT.
\end{theorem}

\begin{proof}
For finite $\Lambda$ with $\#\Lambda\geq 3$, we construct an SFT $\Cong$ on the lamplighter $G:=\reppy\rtimes\ZZ$ in Definition \ref{definition:conformist}. We show that this SFT is nonempty in Proposition \ref{proposition:nonempty} and that it contains no finite orbits in Proposition \ref{proposition:aperiodic}.

To handle the case where $\#\Lambda=2$, we appeal to \cite[Theorem 11]{carpen}, which implies that to find a weakly aperiodic SFT on $G$, it suffices to find one on a finite index subgroup of $G$. The lamplighter $\bigoplus_\ZZ\ZZ/2\ZZ\rtimes\ZZ$ has a finite index subgroup isomorphic to the lamplighter $\bigoplus_\ZZ(\ZZ/2\ZZ\oplus\ZZ/2\ZZ)\rtimes\ZZ$, which has a weakly aperiodic SFT per the above considerations
\end{proof}

\subsection{Organization.} In \S \ref{section:preliminary} we define the lamplighter group $G=\reppy\rtimes\ZZ$ and establish the basic notation used throughout the paper. In \S \ref{section:construction} we construct, for $\#\Lam\geq 3$, an SFT $\Cong\subseteq\zeroone^G$ which we shall show to be weakly aperiodic. In \S \ref{section:nonempty} we show that $\Cong$ is nonempty, and in \S \ref{section:aperiodic} we show that no point of $\Cong$ may have finite index stabilizer in $G$, thus establishing weak aperiodicity. Finally, in \S \ref{section:questions} we discuss possible generalizations and directions for future work.

\subsection{Acknowledgments.} This work was supported by NSF award 1502608. We wish to thank Yongle Jiang for two important corrections.

\section{Preliminaries.}
\label{section:preliminary}
\paragraph{Background notation for lamplighters.} Let $\Lambda$ be a finite group with cardinality $\ell$. We shall need the following notation.
\begin{itemize}
\item For $i\in\ZZ$, let $\Lambda_i$ denote a copy of $\Lambda$, and let $\reppy$ denote the direct sum $\ldots \Lambda_{-1}\oplus\Lambda_0\oplus\Lambda_1\ldots$.
\item Let $\repmin\subseteq\reppy$ be the subgroup generated by $\Lam_{-1},\Lam_{-2}\ldots$.
\item For $\muv\in\reppy$, let $\left(\muv\right)_i\in\Lam$ denote the $i$-th coordinate of $\muv$.
\item For $\lam\in\Lam,$ the element $[\lambda]_i\in\Lambda_i\subseteq\reppy$ is defined so that $\left([\lam]_i\right)_i=\lam$ (and all other coordinates $\left([\lam]_i\right)_j$ of $[\lam]_i$ are equal to the identity $1\in\Lam$).
\end{itemize}

Let $G$ be the lamplighter group $\reppy\rtimes\ZZ$, where the generator $t$ of $\ZZ$ acts by
$$t[\lam]_i t^{-1}=[\lam]_{i+1}.$$
In $G$, it is clear that $t[\lam]_i=[\lam]_{i+1}t$ for any $i\in\ZZ$, and that $[\lam]_i[\lam]_j=[\lam]_j[\lam]_i$ when $i\neq j$. These facts will be used frequently without comment. Note that every element $g\in G$ may be uniquely expressed as $\muv t^k$ for some $\muv\in\reppy$ and $k\in\ZZ$. We will write $1_\Lam$ for the identity of $\Lam$ and $1_G$ for the identity of $G$.

\paragraph{Role models.}
For any $g\in G$, let $\rolemod(g)$ denote the right coset $gt^{-1}\Lam_0$. The letters $\rolemod$ stand for ``role model"---in Definition \ref{definition:conformist} we will define a subshift such that for $\sigma$ in this subshift and $g\in G$, $\sigma(g)$ is always determined by $\sigma|_{\rolemod(g)}$.

\section{Defining the conformist subshift.}
\label{section:construction}
\paragraph{Non-unanimous strict majority.} Given a finite set $F$, and $\sigma\in \{0,1\}^F$, if $\sigma$ satisfies the strict inequalities
$$\frac{1}{2}\#F<\#\{x\in F: \sigma(x)=a\}<\#F$$
for some $a\in\{0,1\}$, we say that $\sigma$ has a non-unanimous strict majority (NUSM) and write
$$\Maj(\sigma)=a.$$
If $\sigma$ satisfies these inequalities for $a=1$, it cannot satisfy them for $a=0$, and vice versa. If $\#F=2$, then no $\sigma\in\{0,1\}^F$ has a NUSM, but if $\#F\geq 3$, then there exist $\sigma\in\{0,1\}^F$ with a NUSM. This is why we assume that $\ell\geq 3$ in the following definition.

\begin{figure}[t]
\labellist
\small\hair 2pt

\pinlabel $g$ at 15 120
\pinlabel $g[\lambda]_1$ at 75 140
\pinlabel $\rolemod(g)=\rolemod(g[\lambda]_1)$ at 75 0

\pinlabel {Forbidden ($\sig|_{\rolemod(g)}$ does not have a NUSM.)} at 280 163
\pinlabel $g$ at 233 142
\pinlabel {Forbidden ($\sig(g)\neq\Maj(\sig|_{\rolemod(g)})$).} at 280 64
\pinlabel $g$ at 265 45

\pinlabel Allowed at 478 163
\pinlabel Allowed at 478 64

\endlabellist

\centering
\centerline{\psfig{file=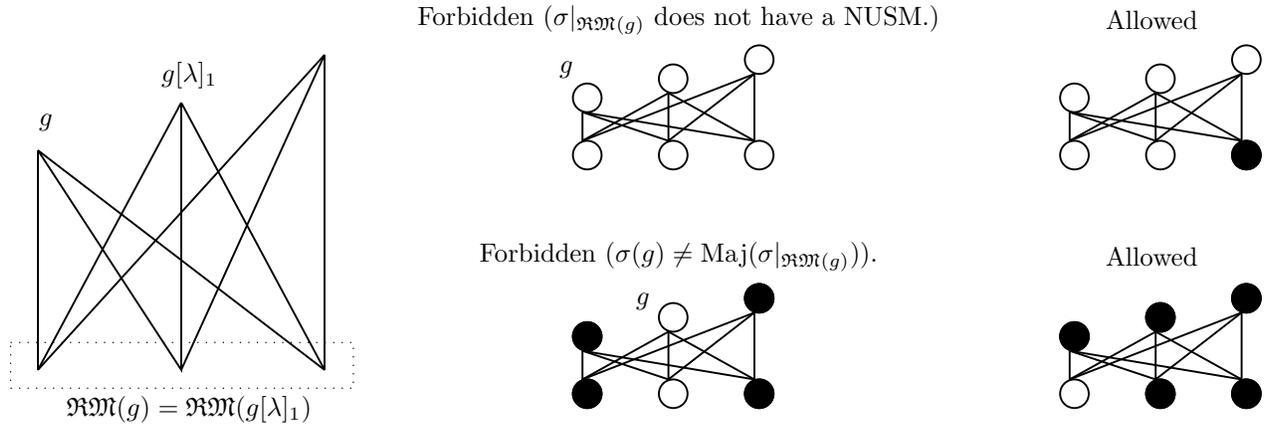,scale=90}}
\caption{At left, we depict the relationship between $g$ and $\rolemod(g)$ in a fragment of the Cayley graph of $G$ with respect to the generating set $\Lam_0 t$. The other columns depict forbidden patterns for $\Cong$ and patterns which may occur in $\Cong$ respectively, where filled and unfilled circles represent $0$ and $1$ respectively.}
\label{figure:conformistdefinition}
\end{figure}

\begin{definition}
\label{definition:conformist}
Suppose that $\ell\geq 3$. The conformist subshift $\Cong\subseteq\{0,1\}^G$ of $G$ is defined to consist of all $\sigma\in\zeroone^G$ such that for all $g\in G$, the restriction $\sigma|_{\rolemod(g)}$ of $\sigma$ to the finite set $\rolemod(g)$ has a NUSM and $\sigma(g)=\Maj\left(\sigma|_{\rolemod(g)}\right)$.
\end{definition}

\paragraph{The conformist subshift is an SFT.} It is evident that $\Cong$ is an SFT with forbidden patterns defined on the set
$$\{1\}\cup \left\{t^{-1}\Lam_0\right\}.$$ Our goal is to show that $\Cong$ is nonempty and weakly aperiodic. In any configuration of $\Cong$, the label of $g\in G$ must conform to the label assumed by most of the role models $\rolemod(g)$ of $g$---whence the name ``conformist subshift". The condition that the labels of $\rolemod(g)$ must be non-unanimous is necessary to prevent fixed points.

\section{Nonemptiness of the conformist subshift.}
\label{section:nonempty}
We first show that $\Cong$ is nonempty by explicitly constructing an element $\sigo$ of $\Cong$. The strategy is as follows.
\begin{itemize}
\item In Definition \ref{definition:nnn}, we define $\nnn:G\To\ZZgo$. In Definition \ref{definition:bell}, we define $b_\ell:\ZZgo\To\zeroone$. The composition of $b_\ell$ and $\nnn$ will give us $\sigo$, our putative element of $\Cong$ (see Definition \ref{definition:sigo}).
\item In Lemma \ref{lemma:nnn}, we show that for any $g\in G$, the restriction $\nnn|_{\rolemod(g)}$ is a bijection onto the set of integers $\left\{\nnn(g)\ell,\ldots,\nnn(g)\ell+\ell-1\right\}$.
\item In Lemma \ref{lemma:bell}, we show that for any $n\in\ZZgo$, the restriction $b_\ell|_{\left\{n\ell,\ldots,n\ell+\ell-1\right\}}$ has a NUSM and $\Maj\left(b_\ell|_{\left\{n\ell,\ldots,n\ell+\ell-1\right\}}\right)=b_\ell(n)$.
\item By combining the previous two lemmas, we show in Proposition \ref{proposition:nonempty} that $\sigo\in\Cong$.
\end{itemize}

\begin{definition}
\label{definition:bell}
For $n\in\ZZgo$, let $b_\ell(n)$ be $0$ if $n$ has an even number of $1$s in its base-$\ell$ expansion, and $1$ if $n$ has an odd number of $1$s in its base-$\ell$ expansion.
\end{definition}

\paragraph{Example.} 
 For $\ell=4$, the first few values of $b_\ell$ are given in the bottom row of following table.

\begin{tabular}{c|ccccccccccccccccc}
$n$ & $0$ & $1$ & $2$ & $3$ & $4$ & $5$
    & $6$ & $7$ & $8$ & $9$ & $10$ & $11$
    & $12$ & $13$ & $14$ & $15$ & $16$  \\
$n$ base-$4$ & $0$ & $1$ & $2$ & $3$ & $10$ & $11$
    & $12$ & $13$ & $20$ & $21$ & $22$ & $23$
    & $30$ & $31$ & $32$ & $33$ & $100$  \\
$b_4(n)$ & $0$ & $1$ & $0$ & $0$ & $1$ & $0$
    & $1$ & $1$ & $0$ & $1$ & $0$ & $0$ & $0$ & $1$ & $0$ & $0$ & $1$  \\
\end{tabular}

\paragraph{Substitutions.} Let $\zeroone^*$ denote the set of finite length words in $\zeroone$, and let $\pil:\zeroone^*\To\zeroone^*$ be the map given by $\pil(0)=010^{\ell-2}$, $\pil(1)=101^{\ell-2}$, and for any other word $w=s_0\ldots s_n$, we have that $\pil(w)$ is the concatenation
$$\pil(s_0)\ldots \pil(s_n).$$
The substitution $\pi_\ell$ is closely related to $b_\ell$. For example, for $\ell=4$, here are the first few iterates of $\pi_\ell$ on the length-$1$ word $0$.
$$\pi_4(0)=0100,$$
$$\pi_4(\pi_4(0))=0100101101000100,$$
$$\pi_4(\pi_4(\pi_4(0)))=0100101101000100
                         1011010010111011
                         0100101101000100
                         0100101101000100.$$
These iterates visibly converge to the infinite word corresponding to $b_4$.

\begin{definition}
\label{definition:nnn}
Fix an enumeration $\Lam=\left\{\alpha_0,\ldots,\alpha_{\ell-1}\right\}$ of $\Lam$, and set $\|\alpha_i\|=i$ for all $\alpha_i\in\Lam$. Recall that any $g\in G$ may be written as $\muv t^k$ for a unique $\muv\in \reppy$ and $k\in \ZZ$. Let $\nnn:G\To\ZZgo$ be defined as follows.
$$\nnn\left(\muv t^k\right)=\sum_{n\geq k}\left\|\left(\muv\right)_n\right\|\ell^{n-k}$$
for any $\muv\in \reppy$ and $k\in \ZZ$.
\end{definition}

The function $\nnn$ is essentially the horizontal coordinate of $g$ in the standard drawing of the Cayley graph of $G$. We now define our alleged element of $\Cong$.

\begin{figure}[t]
\labellist
\small\hair 2pt

\endlabellist

\centering
\centerline{\psfig{file=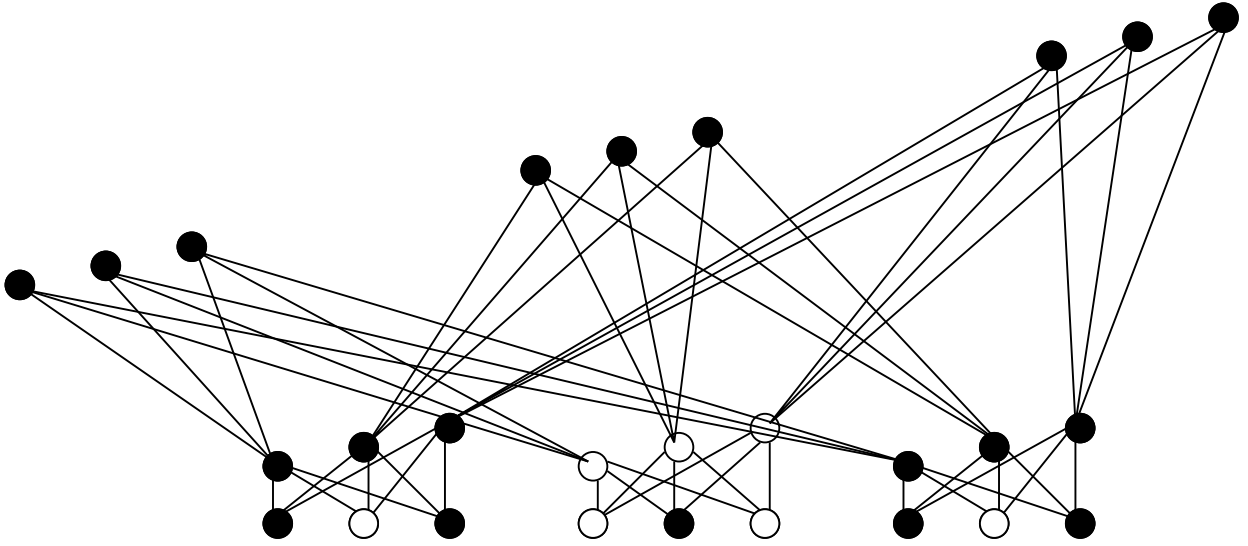,scale=90}}
\caption{A depiction of $\sigo$ when $\ell=3$, drawn on a fragment of the Cayley graph with respect to the generating set $\Lam_0 t$. The vertices depicted are of the form $\muv t^k$ where $\muv\in\Lam_0\oplus\Lam_1\oplus\Lam_2$ and $t\in\{0,1,2\}$. The bottom row depicts vertices in $\Lam_0\oplus\Lam_1\oplus\Lam_2$, arranged in order of increasing $\nnn$ value. Filled and unfilled circles represent $\sigo$ values of $0$ and $1$ respectively.}
\label{figure:conformistconfiguration}
\end{figure}

\begin{definition}
\label{definition:sigo}
Let $\sigo\in\zeroone^G$ be $b_\ell\circ\nnn$.
\end{definition}

In Proposition \ref{proposition:nonempty}, we will show that $\sigo\in\Cong$. We need two ingredients first.

\begin{lemma}
\label{lemma:nnn}
Let $g=\muv t^k$ for $\muv\in\reppy$ and $k\in\ZZ$, and let $\lam\in\Lam$. Then
$$
\nnn\left(gt^{-1}[\lam]_0\right)
=\nnn(g)\ell+\left\|\left(\muv\right)_{k-1}\lam\right\|.$$
Hence, for any $g\in G$, the restriction $\nnn|_{\rolemod(g)}$ is injective with image $\left\{\nnn(g)\ell,\ldots,\nnn(g)\ell+\ell-1\right\}$.
\end{lemma}

\begin{proof}
We compute
$$\nnn\left(gt^{-1}[\lam]_0\right)=\nnn\left(\muv t^k t^{-1} [\lam]_0\right)=\nnn\left(\muv t^{k-1}[\lam]_0\right)=\nnn\left(\muv [\lam]_{k-1} t^{k-1}\right)$$
$$
=\sum_{n\geq k-1}\left\|\left(\muv[\lam]_{k-1}\right)_{n}\right\|\ell^{n-k+1}
=\sum_{n\geq k}\left\|\left(\muv\right)_n\right\|\ell^{n-k+1}+\left\|\left(\muv[\lam]_{k-1}\right)_{k-1}\right\|\ell^0$$
$$=\nnn(g)\ell+\left\|\left(\muv\right)_{k-1}\lam\right\|.$$

As $\lam$ ranges over $\Lam$, $\left\|\left(\muv\right)_{k-1}\lam\right\|$ takes on each value of $\ellint$ exactly once, and we obtain the desired result for the restriction $\nnn|_{\rolemod(g)}$
\end{proof}

\begin{lemma}
\label{lemma:bell}
For $n\in\ZZgo$ and $j\in\ellint$, $b_\ell(n\ell+j)$ is equal to the $j$-th letter of $\pil(b_\ell(n))$. 
\end{lemma}

\begin{proof}
If $j\neq 1$, then $n\ell+j$ and $n$ have the same number of $1$s in their base-$\ell$ expansion, so $b(n\ell+j)=b(n)$ which equals the $j$-th letter of $\pil(b(n))$ by definition of $\pil$.

On the other hand, if $j=1$, then $n\ell+j$ has one more $1$ than $n$ in its base-$\ell$ expansion. By looking at the definition of $\pil$ and $b_\ell$, it follows again that $b(n\ell+j)$ is the $j$-th letter of $\pil(b(n))$ (as both are equal to the element of $\zeroone$ which is not $b(n)$).
\end{proof}

Combining the preceding lemmas, we now obtain the desired result that $\Cong$ is nonempty.

\begin{proposition}
\label{proposition:nonempty}
For $\sigo$ as defined in Definition \ref{definition:sigo}, we have $\sigo\in\Cong$.
\end{proposition}

\begin{proof}
We must show that for all $g\in G$, the restriction $\sigo|_{\rolemod(g)}$ has a NUSM and $$\Maj\left(\sigo|_\rolemod(g)\right)=\sigo(g).$$

By abuse of notation, view the words $\pil(0)$ and $\pil(1)$ as functions $\ellint\To\zeroone$. It is clear that both have a NUSM, and that $\Maj\left(\pil(a)\right)=a$ for $a\in\zeroone$.

For any $g\in G$, by Lemma \ref{lemma:nnn}, $\nnn$ restricts to a bijection
$$\rolemod(g)\To\left\{\ell\nnn(g),\ldots,\ell\nnn(g)+\ell-1\right\}.$$ By Lemma \ref{lemma:bell}, the restriction $b_\ell|_{\left\{\ell\nnn(g),\ldots,\ell\nnn(g)+\ell-1\right\}}$ has a NUSM and
$$\Maj\left(b_\ell|_{\left\{\ell\nnn(g),\ldots,\ell\nnn(g)+\ell-1\right\}}\right)=b_\ell(\nnn(g)).$$ Combining these observations, it follows that $\sigo|_{\rolemod(g)}$ has a NUSM and $\Maj\left(\sigo|_{\rolemod(g)}\right)=\sigo(g)$. Thus, $\sigo\in\Cong$ as desired.
\end{proof}

\section{Weak aperiodicity of the conformist subshift.}
\label{section:aperiodic}
We now see that no element of the conformist subshift $\Cong$ can be fixed by a finite index subgroup of $G$. The structure of the proof is as follows.
\begin{itemize}
\item Lemma \ref{lemma:constant} establishes the key fact that for any $\sig\in\Cong$, $g\in G$ and $\muv\in\repmin$, we have $\sig\left(g\muv\right)=\sig(g)$.
\item Lemma \ref{lemma:product} shows that if $\Gam$ is a finite index normal subgroup of $G$, and $L=\Gam\cap\reppy$, then any $\muv\in\reppy$ can be written as $\muv_L\muv_-$ for some $\muv_L\in L$ and $\muv_-\in\repmin$.
\item Proposition \ref{proposition:aperiodic} combines these two facts with the definition of $\Cong$ to show that no element of $\Cong$ is fixed by a finite index subgroup of $G$.
\end{itemize}

\begin{lemma}
\label{lemma:constant}
Let $n\in\NN$. For $\sig\in\Cong$, $g\in G$, and $\muv\in\Lam_{-1}\Lam_{-2}\ldots\Lam_{-n}$, we have
$$\sig\left(g\muv\right)=\sig(g).$$
\end{lemma}

\begin{proof}
We proceed by induction on $n$. For $g\in G$ and $\lam\in \Lam$,
$$\rolemod\left(g[\lam]_{-1}\right)=g[\lam]_{-1}t^{-1}\Lam_0=gt^{-1}[\lam]_0\Lam_0=gt^{-1}\Lam_0=\rolemod(g),$$
so
$$\sig(g)=\Maj\left(\sig|_{\rolemod(g)}\right)=\Maj\left(\sig|_{\rolemod\left(g[\lam]_{-1}\right)}\right)=\sig\left(g[\lam]_{-1}\right).$$
This establishes, as a base case, that $\sig\left(g\muv\right)=\sig(g)$ for any $g\in G$ and $\muv\in\Lam_{-1}$.

Suppose by inductive hypothesis that $\sig\left(g\muv\right)=\sig(g)$ for all $g\in G$ and $\muv\in\Lam_{-1}\ldots\Lam_{-n}$. We must show that the same equality holds for any $g\in G$ and $\muv\in\Lam_{-1}\ldots\Lam_{-(n+1)}$. It suffices to consider the case where $\muv=[\lam]_{-(n+1)}$ for some $\lam\in\Lam$. We see that
$$\rolemod\left(g[\lam]_{-(n+1)}\right)=g[\lam]_{-(n+1)}t^{-1}\Lam_0
=gt^{-1}\Lam_0 [\lam]_{-n}=\rolemod(g)[\lam]_{-n}.$$
\noindent For any $g'\in\rolemod(g)$, our inductive hypothesis implies that $\sig\left(g'[\lam]_{-n}\right)=\sig(g')$. It follows that
$$\sig\left(g[\lam]_{-(n+1)}\right)=\Maj\left(\sig|_{\rolemod\left(g[\lam]_{-(n+1)}\right)}\right)=\Maj\left(\sig|_{\rolemod(g)[\lam]_{-n}}\right)=\Maj\left(\sig|_{\rolemod(g)}\right)=\sig(g)$$
as desired.
\end{proof}

\begin{lemma}
\label{lemma:product}
Let $\Gam$ be a finite index normal subgroup of $G$, and let $L=\Gam\cap\reppy$. Then every $\muv\in\reppy$ is equal to some product $\muv_L\muv_-$ where $\muv_L\in L$ and $\muv_-\in\repmin$.
\end{lemma}

\begin{proof}
Since $L$ is normal in $\reppy$, $t$ acts by automorphisms on the finite group $\reppy/L$. Let $d$ be the order of $t$ in $\Aut\left(\reppy/L\right)$.

For any $\muv\in\reppy$, there is some $k$ such that $t^{-kd}\muv t^{kd}\in\repmin$. Let $\muv_- =t^{-kd}\muv t^{kd}$ for some such $k$. Since $t^d$ acts trivially on $\reppy/L$, we must have that $\muv$ lies in the coset $L\muv_-$, so that $\muv=\muv_L\muv_-$ for some $\muv_L\in L$ as desired.
\end{proof}

\begin{proposition}
\label{proposition:aperiodic}
A finite index subgroup $\Gam\subseteq G$ cannot fix a configuration $\sig\in\Cong$.
\end{proposition}

\begin{proof}
For any $\sig\in\Cong$, and $\muv_-\in\repmin$, by Lemma \ref{lemma:constant}, $\sig\left(\muv_-\right)=\sig(1_G)$. If $\sig$ is fixed by a finite index subgroup of $G$, then it is also fixed by a finite index normal subgroup $\Gam\subseteq G$, and hence by $L=\Gam\cap\reppy$. By Lemma \ref{lemma:product}, every $\muv\in\reppy$ may be represented as $\muv_L\muv_-$ where $\muv_L\in L$ and $\muv_-\in\repmin$, and hence, as $\sig=\sig\cdot\muv_L$, we have
$$\sig\left(\muv\right)=\left(\sig\cdot\muv_L\right)\left(\muv_-\right)=\sig\left(\muv_L\muv_-\right)=\sig\left(\muv_-\right)=\sig(1_G).$$

In particular, this holds for $\muv=[\lam]_0$, where $\lam\in\Lam$. Thus, $\sig|_{\Lam_0}$ does not have a NUSM (as the above calculation shows that it is unaninimous---i.e., $\sig$ takes all elements of $\Lam_0$ to the same element of $\zeroone$). But by definition of $\Cong$, the restriction of $\sig$ to $\Lam_0=\rolemod(t)$ must have a NUSM.
\end{proof}

\section{Further questions and possible generalizations.}
\label{section:questions}
A nonempty SFT $X\subseteq A^G$ is said to be {\bf strongly aperiodic} if for all $\sig\in X$, the stabilizer $\Stab_G(\sig)$ is equal to the trivial subgroup $\{1_G\}$ of $G$. The conformist subshift $\Cong$ is not strongly aperiodic. In fact, the configuration $\sigo$ of Definition \ref{definition:sigo} was shown to be in $\Cong$ in Proposition \ref{proposition:nonempty} and satisfies $\sigo\cdot t=\sigo$, as we shall now see.

Because $\sigo$ is defined as $b_\ell\circ \nnn$, it suffices to show that $\nnn(tg)=\nnn(g)$ for any $g\in G$. Given $g\in G$, write $g$ as $\muv t^k$ for $\muv\in \reppy$ and $k\in \ZZ$.  Observe that $\left(t\muv t^{-1}\right)_n=\left(\muv\right)_{n-1}$ for any $n\in\ZZ$. We now calculate
$$\nnn(tg)=\nnn\left(t\muv t^k\right)=\nnn\left(\left(t\muv t^{-1}\right)t^{k+1}\right)$$
$$=\sum_{n\geq k+1}\left\|\left(t\muv t^{-1}\right)_n\right\|\ell^{n-(k+1)}
=\sum_{(n-1)\geq k}\left\|\left(\muv\right)_{n-1}\right\|\ell^{(n-1)-k}
=\nnn(g)
$$
as desired.

Hence, the following question remains open.


\paragraph{Question.} Do lamplighters have strongly aperiodic SFT?\\

The domino problem for $G$ asks whether one can determine from a finite set of forbidden patterns whether the SFT they carve out is empty. One might expect that it should be possible to use our techniques to encode the dynamics of a piecewise rational affine map into forbidden patterns on a lamplighter, as done in \cite{ak} to show that Baumslag-Solitar groups do not have decidable domino problem. However, some careful thought reveals that it is not at all obvious that this is actually possible, so the following question remains open.

\paragraph{Question.} Do lamplighters have decidable domino problem?\\

We end on a pessimistic note. Our techniques likely extend to other wreath products of $\ZZ$ with amenable groups, but such wreath products are already known to have weakly aperiodic SFT. Moreover, there seems to be no way to extend our methods to more general groups which surject onto $\ZZ$. Indeed, it is hard to see how any analogue of Lemma \ref{lemma:constant} could be true except in a wreath product. Hence, although an infinite finitely presented group with no weakly aperiodic SFT must virtually surject onto $\ZZ$, we are no closer to verifying that the Carroll-Penland conjecture holds for finitely presented groups. On the other hand, we do not know a single candidate for a counterexample to the Carroll-Penland conjecture.

\bibliographystyle{plain}
\bibliography{bibliography}
\end{document}